\documentclass[final,preprint,3p,sort&compress]{elsarticle}
\journal{{\tt arXiv.org}}

\usepackage[dvips]{epsfig}
\usepackage{graphicx} 
\usepackage{pdfpages}
\usepackage{latexsym}
\usepackage{verbatim}
\usepackage{amsmath,amsthm}
\usepackage{amssymb}
\usepackage{bbm}
\usepackage{fonts}
\usepackage{enumerate}
\usepackage{placeins}
\usepackage{standalone}
\usepackage{paralist}
\usepackage{subcaption}

\definecolor{greenyellow}   {cmyk}{0.15, 0   , 0.69, 0   }
\definecolor{yellow}        {cmyk}{0   , 0   , 1   , 0   }
\definecolor{goldenrod}     {cmyk}{0   , 0.10, 0.84, 0   }
\definecolor{dandelion}     {cmyk}{0   , 0.29, 0.84, 0   }
\definecolor{apricot}       {cmyk}{0   , 0.32, 0.52, 0   }
\definecolor{peach}         {cmyk}{0   , 0.50, 0.70, 0   }
\definecolor{melon}         {cmyk}{0   , 0.46, 0.50, 0   }
\definecolor{yelloworange}  {cmyk}{0   , 0.42, 1   , 0   }
\definecolor{orange}        {cmyk}{0   , 0.61, 0.87, 0   }
\definecolor{burntorange}   {cmyk}{0   , 0.51, 1   , 0   }
\definecolor{bittersweet}   {cmyk}{0   , 0.75, 1   , 0.24}
\definecolor{redorange}     {cmyk}{0   , 0.77, 0.87, 0   }
\definecolor{mahogany}      {cmyk}{0   , 0.85, 0.87, 0.35}
\definecolor{maroon}        {cmyk}{0   , 0.87, 0.68, 0.32}
\definecolor{brickred}      {cmyk}{0   , 0.89, 0.94, 0.28}
\definecolor{red}           {cmyk}{0   , 1   , 1   , 0   }
\definecolor{orangered}     {cmyk}{0   , 1   , 0.50, 0   }
\definecolor{rubinered}     {cmyk}{0   , 1   , 0.13, 0   }
\definecolor{wildstrawberry}{cmyk}{0   , 0.96, 0.39, 0   }
\definecolor{salmon}        {cmyk}{0   , 0.53, 0.38, 0   }
\definecolor{carnationpink} {cmyk}{0   , 0.63, 0   , 0   }
\definecolor{magenta}       {cmyk}{0   , 1   , 0   , 0   }
\definecolor{violetred}     {cmyk}{0   , 0.81, 0   , 0   }
\definecolor{rhodamine}     {cmyk}{0   , 0.82, 0   , 0   }
\definecolor{mulberry}      {cmyk}{0.34, 0.90, 0   , 0.02}
\definecolor{redviolet}     {cmyk}{0.07, 0.90, 0   , 0.34}
\definecolor{fuchsia}       {cmyk}{0.47, 0.91, 0   , 0.08}
\definecolor{lavender}      {cmyk}{0   , 0.48, 0   , 0   }
\definecolor{thistle}       {cmyk}{0.12, 0.59, 0   , 0   }
\definecolor{orchid}        {cmyk}{0.32, 0.64, 0   , 0   }
\definecolor{darkorchid}    {cmyk}{0.40, 0.80, 0.20, 0   }
\definecolor{purple}        {cmyk}{0.45, 0.86, 0   , 0   }
\definecolor{plum}          {cmyk}{0.50, 1   , 0   , 0   }
\definecolor{violet}        {cmyk}{0.79, 0.88, 0   , 0   }
\definecolor{royalpurple}   {cmyk}{0.75, 0.90, 0   , 0   }
\definecolor{blueviolet}    {cmyk}{0.86, 0.91, 0   , 0.04}
\definecolor{periwinkle}    {cmyk}{0.57, 0.55, 0   , 0   }
\definecolor{cadetblue}     {cmyk}{0.62, 0.57, 0.23, 0   }
\definecolor{cornflowerblue}{cmyk}{0.65, 0.13, 0   , 0   }
\definecolor{midnightblue}  {cmyk}{0.98, 0.13, 0   , 0.43}
\definecolor{navyblue}      {cmyk}{0.94, 0.54, 0   , 0   }
\definecolor{royalblue}     {cmyk}{1   , 0.50, 0   , 0   }
\definecolor{blue}          {cmyk}{1   , 1   , 0   , 0   }
\definecolor{cerulean}      {cmyk}{0.94, 0.11, 0   , 0   }
\definecolor{cyan}          {cmyk}{1   , 0   , 0   , 0   }
\definecolor{processblue}   {cmyk}{0.96, 0   , 0   , 0   }
\definecolor{skyblue}       {cmyk}{0.62, 0   , 0.12, 0   }
\definecolor{turquoise}     {cmyk}{0.85, 0   , 0.20, 0   }
\definecolor{tealblue}      {cmyk}{0.86, 0   , 0.34, 0.02}
\definecolor{aquamarine}    {cmyk}{0.82, 0   , 0.30, 0   }
\definecolor{bluegreen}     {cmyk}{0.85, 0   , 0.33, 0   }
\definecolor{emerald}       {cmyk}{1   , 0   , 0.50, 0   }
\definecolor{junglegreen}   {cmyk}{0.99, 0   , 0.52, 0   }
\definecolor{seagreen}      {cmyk}{0.69, 0   , 0.50, 0   }
\definecolor{green}         {cmyk}{1   , 0   , 1   , 0   }
\definecolor{forestgreen}   {cmyk}{0.91, 0   , 0.88, 0.12}
\definecolor{pinegreen}     {cmyk}{0.92, 0   , 0.59, 0.25}
\definecolor{limegreen}     {cmyk}{0.50, 0   , 1   , 0   }
\definecolor{yellowgreen}   {cmyk}{0.44, 0   , 0.74, 0   }
\definecolor{springgreen}   {cmyk}{0.26, 0   , 0.76, 0   }
\definecolor{olivegreen}    {cmyk}{0.64, 0   , 0.95, 0.40}
\definecolor{rawsienna}     {cmyk}{0   , 0.72, 1   , 0.45}
\definecolor{sepia}         {cmyk}{0   , 0.83, 1   , 0.70}
\definecolor{brown}         {cmyk}{0   , 0.81, 1   , 0.60}
\definecolor{tan}           {cmyk}{0.14, 0.42, 0.56, 0   }
\definecolor{gray}          {cmyk}{0   , 0   , 0   , 0.50}
\definecolor{black}         {cmyk}{0   , 0   , 0   , 1   }
\definecolor{white}         {cmyk}{0   , 0   , 0   , 0   } 

 \usepackage[]{hyperref} 
        
\usepackage{pgfplots}
\pgfplotsset{compat=newest}       

\newcommand{\externaltikz}[2]{\includegraphics{Externals/#1}}

\usepackage{relinput}
\newtheorem{theorem}{Theorem}[section]
\newtheorem{definition}[theorem]{Definition}
\newtheorem{remark}[theorem]{Remark}

\newcounter{tikzsubfigcounter}[figure]
\renewcommand{\thetikzsubfigcounter}{\the\numexpr\value{figure}+1\relax\alph{tikzsubfigcounter}}

\newcounter{tikzsubfigcounterinvisible}[figure]
\renewcommand{\thetikzsubfigcounterinvisible}{\the\numexpr\value{figure}+1\relax\alph{tikzsubfigcounterinvisible}}

\newcommand{\settikzlabel}[1]{ %
\refstepcounter{tikzsubfigcounterinvisible} \label{#1} 
}

\numberwithin{equation}{section}

\title{Second-order mixed-moment model with differentiable ansatz function in slab geometry}

\author[fs]{Florian Schneider}
\address[fs]{Fachbereich Mathematik, TU Kaiserslautern, Erwin-Schr\"odinger-Str., 67663 Kaiserslautern, Germany, {\tt schneider@mathematik.uni-kl.de}}

\date{}

\usepgfplotslibrary{groupplots}

\newlength{\figureheight}
\newlength{\figurewidth}
\usetikzlibrary{arrows,shapes,positioning}
\usetikzlibrary{decorations.markings}
\tikzstyle arrowstyle=[scale=1]
\tikzstyle directed=[postaction={decorate,decoration={markings,
		mark=at position .65 with {\arrow[arrowstyle]{stealth}}}}]
\tikzstyle reverse directed=[postaction={decorate,decoration={markings,
		mark=at position .65 with {\arrowreversed[arrowstyle]{stealth};}}}]

\newcommand{\figref}[1]{Figure~\ref{#1}}

\newcommand{\abs}[1]{\ensuremath{\left| #1 \right|}}

\newcommand{\R}{\mathbb{R}}
\newcommand{\Rpos}{\R_{\geq 0}}
\newcommand{\scattering}{\ensuremath{\sigma_s}}
\newcommand{\absorption}{\ensuremath{\sigma_a}}

\newcommand{\source}{\ensuremath{Q}}

\newcommand{\sphere}[1][2]{\ensuremath{\mathcal{S}^{#1}}}

\newcommand{\distribution}[1][ ]{\ensuremath{\psi_{#1}}}
\newcommand{\distributiontzero}{\ensuremath{\distribution[\timevar=0]}}
\newcommand{\distributionboundary}{\ensuremath{\distribution[b]}}
\newcommand{\distributionvacuum}{\ensuremath{\distribution[\text{vac}]}}
\newcommand{\ansatz}[1][ ]{\ensuremath{\hat{\psi}_{#1}}}

\newcommand{\momentorder}{\ensuremath{N}}
\newcommand{\momentnumber}{\ensuremath{n}}
\newcommand{\basis}[1][ ]{{\ensuremath{\bb_{#1}}}} 
\newcommand{\basisind}{\ensuremath{i}} 
\newcommand{\basiscomp}[1][\basisind]{\ensuremath{b_{#1}}} 
\newcommand{\SHl}{{\ensuremath{l}}} 

\newcommand{\moments}[1][ ]{\ensuremath{\bu_{#1}}} 
\newcommand{\momentcomp}[1]{\ensuremath{u_{#1}}} 

\newcommand{\normalizedmoments}[1][ ]{\ensuremath{\bsphi_{#1}}} 
\newcommand{\normalizedmomentcomp}[1]{\ensuremath{\phi_{#1}}} 
\newcommand{\normalizedisotropicmoment}{\normalizedmoments[\text{iso}]}
\newcommand{\multipliers}[1][ ]{\ensuremath{\bsalpha_{#1}}} 

\newcommand{\SCheight}{\ensuremath{\mu}} 
\newcommand{\Domain}{\ensuremath{X}} 
\newcommand{\timeint}{\ensuremath{T}} 
\newcommand{\tf}{\ensuremath{t_f}} 
\newcommand{\timevar}{\ensuremath{t}} 

\newcommand{\ints}[1]{\ensuremath{\left<#1\right>}}
\newcommand{\intA}[2]{\ensuremath{\left<#1\right>_{#2}}}

\newcommand{\collisionop}{\ensuremath{\cC}}
\newcommand{\collision}[1]{\ensuremath{\collisionop\left(#1\right)}}

\newcommand{\dirac}{\ensuremath{\delta}}

\newcommand{\indicator}[1]{\ensuremath{\mathbbm{1}_{#1}}}

\newcommand{\Lp}[1]{\ensuremath{L_{#1}}}
\newcommand{\RD}[2]{\ensuremath{\mathcal{R}_{#1}^{#2}}}
\newcommand{\RDone}[1]{\left. \RD{#1}{} \right|_{\density = 1}}
\newcommand{\dRDone}[1]{\left. \partial\RD{#1}{} \right|_{\density = 1}}

\newcommand{\AnsatzSpace}{\ensuremath{\cA}}

\newcommand{\PN}[1][\momentorder]{\ensuremath{\text{P}_{#1}}}
\newcommand{\MN}[1][\momentorder]{\ensuremath{\text{M}_{#1}}}

\newcommand{\MMN}[1][\momentorder]{\ensuremath{\text{MM}_{#1}}}
\newcommand{\DMMN}[1][\momentorder]{\ensuremath{\text{DMM}_{#1}}}

\newcommand{\SPN}[1][\momentorder]{\ensuremath{\text{SP}_{#1}}}

\newcommand{\Flux}{\ensuremath{\bF}}
\newcommand{\Source}{\ensuremath{\bs}}

\newcommand{\optJacobian}{\ensuremath{\bJ}}

\newcommand{\optHessian}{\ensuremath{\bH}}

\newcommand{\eigenvalue}{\ensuremath{\lambda}}

\DeclareMathOperator*{\argmin}{argmin}
\newcommand{\ld}[1]{\ensuremath{{#1}_*}} 
\newcommand{\entropy}{\ensuremath{\eta}} 
\newcommand{\entropyFunctional}{\ensuremath{\mathcal{H}}} 

\def\quand{\quad \mbox{and} \quad}

\newcommand{\x}{\ensuremath{x}}
\newcommand{\y}{\ensuremath{y}}
\newcommand{\z}{\x}
\newcommand{\dx}{\partial_{\x}}

\newcommand{\dz}{\partial_{\z}}
\newcommand{\dt}{\partial_\timevar}

\newcommand{\intp}[1]{\intA{#1}{+}}
\newcommand{\intpm}[1]{\intA{#1}{\pm}}
\newcommand{\intm}[1]{\intA{#1}{-}}

\newcommand{\LaplaceBeltramiProjection}{\ensuremath{\Delta_\SCheight}} 

\newcommand{\regularizationParameter}[1][ ]{\ensuremath{r_{#1}}}

\newcommand{\zL}{\ensuremath{\z_{L}}}
\newcommand{\zR}{\ensuremath{\z_{R}}}

\newcommand{\density}{\ensuremath{\rho}}

\pgfplotscreateplotcyclelist{color parula}{%
	royalblue!70,every mark/.append style={solid,line width = 0pt,fill=royalblue!60!black},mark=ball\\%
	goldenrod!50!yelloworange,every mark/.append style={solid,fill=goldenrod!30!black},mark=*\\%
	green,every mark/.append style={black,fill=yellowgreen},mark=diamond*\\%
	bluegreen,every mark/.append style={fill=black},mark=triangle*\\%
	royalblue!70,densely dashed,every mark/.append style={solid,line width = 0pt,fill=royalblue!60!black},mark=ball\\%
	goldenrod!50!yelloworange,densely dashed,every mark/.append style={solid,black,fill=goldenrod},mark=diamond*\\%
	yellowgreen,densely dashed,every mark/.append style={solid,fill=yellowgreen!60!royalblue},mark=square*, mark size= 1pt\\%
	bluegreen,densely dashed,mark size = 1.5pt,every mark/.append style={solid,fill=white},mark=otimes*\\
	royalblue,densely dashed,mark size = 1.5pt,every mark/.append style={solid,fill=royalblue!30!black},mark=pentagon*\\
	goldenrod!40!yelloworange,every mark/.append style={solid,fill=goldenrod!30!black},mark=|\\%
	}

\usepackage{media9}

\begin{document}

\begin{abstract}
We study differentiable mixed-moment models (full zeroth and first moment, half higher moments) for a Fokker-Planck equation in one space dimension. Mixed-moment minimum-entropy models are known to overcome the zero net-flux problem of full-moment minimum entropy $\MN$ models. Realizability theory for these modification of mixed moments is derived for second order. Numerical tests are performed with a kinetic first-order finite volume scheme and compared with $\MN$, classical $\MMN$ and a $\PN$ reference scheme.
\end{abstract}
\maketitle

\noindent


\section{Introduction}
%

We investigate time-dependent kinetic transport equations like the Fokker--Planck equation, arising from the Boltzmann equation \cite{cercignani2012boltzmann,Boltzmann1872} under the assumption of extremely forward-peaked scattering \cite{Pom92}. They describe the propagation of ``radiation particles'' like photons or electrons which travel at time $\timevar$ from their current position in a specific direction and how they interact with the surrounding matter. Without any assumptions or dimensional reductions this typically leads to a six- or seven-dimensional state space. Applications reach from electron transport in solids and plasmas, neutron transport in nuclear reactors, photon transport in superfluids and radiative transfer to the context of biological modelling, e.g. for studying cell movement (chemotaxis/haptotaxis) or wolf migration \cite{Hadeler2000,Hillen2013,Chalub2004}.

A common approach to reduce the dimensionality is given by the method of moments \cite{Eddington,Levermore1998}, which is a class of Galerkin methods for the approximation of such time-dependent kinetic transport equations. One chooses a set of angular basis functions, tests the kinetic equation against it and integrates over the angular variable, removing the angular dependence while getting a (potentially huge) system of differential equations in space and time. Well-known examples are the classical $\PN$ methods \cite{Jea17,Eddington,Brunner2005b}, their simplifications, the $\SPN$ \cite{Gel61} methods and entropy minimization $\MN$ models \cite{Min78,DubFeu99,BruHol01,Monreal2008,AllHau12}. Especially the latter is favourable since the moment equations are always closed with a positive ansatz function, respecting the positivity of the kinetic distribution to be approximated. In many situations these models perform very well, but since they result from averaging over the complete velocity space, they can produce physically wrong steady-state shocks. It has been shown by Hauck \cite{Hauck2010} that these shocks exist for every odd order. 

To improve this situation, half- or partial-moment models were introduced in \cite{DubKla02,Frank2006}. These models work especially well in one space dimension since they capture the potential discontinuity of the probability density in the angular variable which in 1D is well-located. Unfortunately, in a Fokker-Planck operator is used instead of the standard integral-scattering operator (BGK type), these half-moment approximations fail significantly. A reason for this is that the domain of definition of the Laplace-Beltrami operator requires continuous functions in one dimension. \cite{Schneider2014}.
 
An intermediate model respecting the continuity of full-moment models while allowing the flexibility of partial moments is the mixed-moment model, which was proposed in \cite{Frank07,Schneider2014,Schneider2015c}. Contrary to a typical half-moment approximation, the lowest order moment (density) is kept as a full moment while all higher moments are half moments. 

Although these $\MMN$ models satisfy the above-mentioned property of having a continuous ansatz function, the numerical discretization of it is highly non-trivial due to the appearance of microscopic terms (i.e. the moments of the Laplace-Beltrami operator depend on the values of the ansatz itself). Especially in multiple dimensions, naive implementations fail at discretizing the (semi-)microscopic quantities (line integrals over quadrant/octant boundaries) \cite{Schneider2016,Schneider2015c}. To overcome this (numerical) problem, we investigate a modification of the mixed-moment model. This new $\DMMN$ model has more regularity, i.e. its ansatz is differentiable, resulting in a more robust numerical implementation while maintaining most of the benefits of the classical $\MMN$ model.

The first part of the paper shortly reviews the method of moments and the minimum-entropy ansatz. Afterwards, the concept of realizability (the fact, that a moment vector is associated with a non-negative distribution function) is introduced and a concrete characterization of the realizable set for the $\DMMN[2]$ model is derived. Furthermore, the eigenstructure of this model is explored. Then, the performance of the new model is investigated in two benchmark tests, showing that the $\DMMN[2]$ is competitive compared to $\MN$ and $\MMN$ models with the same number of degrees of freedom. The paper is concluded by a summary and an outlook on future work.
\section{Models}
In slab geometry, the transport equation under consideration for the particle distribution $\distribution = \distribution(\timevar,\z,\SCheight)$ has the form 
\begin{align}
\label{eq:TransportEquation1D}
\dt\distribution+\SCheight\dz\distribution + \absorption\distribution = \scattering\collision{\distribution}+\source, \qquad \timevar\in\timeint,\z\in\Domain,\SCheight\in[-1,1].
\end{align}
The physical parameters are the absorption and scattering coefficient $\absorption,\scattering:\timeint\times\Domain\to\Rpos$, respectively, and the emitting source $\source:\timeint\times\Domain\times[-1,1]\to\Rpos$. 

Collision of particles is modelled by the Laplace-Beltrami operator $$\collision{\distribution} = \frac12 \LaplaceBeltramiProjection \distribution = \frac12 \cfrac{d}{d\SCheight}\left(\left(1-\SCheight^2\right)\cfrac{d\distribution}{d\SCheight}\right).$$
This operator appears, for example, as the result of an asymptotic analysis of the Boltzmann equation under the assumption of small energy loss and deflection, and forward-peaked scattering in the context of electron transport \cite{Frank07,Pom92,HenIzaSie06}.

The transport equation \eqref{eq:TransportEquation1D} is supplemented by initial and boundary conditions:
\begin{subequations}
\begin{align}
\distribution(0,\z,\SCheight) &= \distributiontzero(\z,\SCheight) &\text{for } \z\in\Domain = (\zL,\zR), \SCheight\in[-1,1], \label{eq:TransportEquation1DIC}\\
\distribution(\timevar,\zL,\SCheight) &= \distributionboundary(\timevar,\zL,\SCheight) &\text{for } \timevar\in\timeint, \SCheight>0,  \label{eq:TransportEquation1DBCa}\\
\distribution(\timevar,\zR,\SCheight) &= \distributionboundary(\timevar,\zR,\SCheight) &\text{for } \timevar\in\timeint, \SCheight<0. \label{eq:TransportEquation1DBCb}
\end{align}
\end{subequations}

In general, solving equation \eqref{eq:TransportEquation1D} is very expensive in two and three dimensions due to the high dimensionality of the state space. 

For this reason it is convenient to use some type of spectral or Galerkin method to transform the high-dimensional equation into a system of lower-dimensional equations. Typically, one chooses to reduce the dimensionality by representing the angular dependence of $\distribution$ in terms of some basis $\basis$.
\begin{definition}
The vector of functions $\basis:[-1,1]\to\R^{\momentnumber}$ consisting of $\momentnumber$ basis functions $\basiscomp[\basisind]$, $\basisind=0,\ldots\momentnumber-1$ of maximal \emph{order} $\momentorder$ is called an \emph{angular basis}.\\
The so-called \emph{moments} of a given distribution function $\distribution$ with respect to $\basis$ are then defined by
\begin{align}
\label{eq:moments}
\moments =\ints{{\basis}\distribution} = \left(\momentcomp{0},\ldots,\momentcomp{\momentnumber-1}\right)^T,
\end{align}
where the integration $\ints{\cdot} := \int\limits_{-1}^1\cdot~d\SCheight$ is performed componentwise.\\
Assuming for simplicity $\basiscomp[0]\equiv 1$, the quantity $\density := \momentcomp{0} = \ints{\basiscomp[0]\distribution}=\ints{\distribution}$ is called \emph{local particle density}. 
Furthermore, \emph{normalized moments} $\normalizedmoments = \left(\normalizedmomentcomp{1},\ldots,\normalizedmomentcomp{\momentnumber-1}\right)\in\R^{\momentnumber-1}$ are defined as 
\begin{align}
\label{eq:NormalizedMoments}
\normalizedmomentcomp{\basisind} = \cfrac{\momentcomp{\basisind}}{\momentcomp{0}}~, \qquad \basisind=1,\ldots\momentnumber-1.
\end{align}
\end{definition}
To obtain a set of equations for $\moments$, \eqref{eq:TransportEquation1D} has to be multiplied through by $\basis$ and integrated over $[-1,1]$, giving
\begin{align*}
\ints{\basis\dt\distribution}+\ints{\basis\dz\SCheight\distribution} + \ints{\basis\absorption\distribution} = \scattering\ints{\basis\collision{\distribution}}+\ints{\basis\source}.
\end{align*}
Collecting known terms, and interchanging integrals and differentiation where possible, the moment system has the form
\begin{align}
\label{eq:MomentSystemUnclosed1D}
\dt\moments+\dz\ints{\SCheight \basis\distribution} + \absorption\moments = \scattering\ints{\basis\collision{\distribution}}+\ints{\basis\source}.
\end{align}

The solution of \eqref{eq:MomentSystemUnclosed1D} is equivalent to the one of \eqref{eq:TransportEquation1D} if $\basis$ is a basis of $\Lp{2}(\sphere,\R)$. 

Since it is impractical to work with an infinite-dimensional system, only a finite number of $\momentnumber<\infty$ basis functions $\basis$ of order $\momentorder$ can be considered. Unfortunately, there always exists an index $\basisind\in\{0,\dots,\momentnumber-1\}$ such that the components of $\basiscomp\cdot\SCheight$ are not in the linear span of $\basis$. Therefore, the flux term cannot be expressed in terms of $\moments$ without additional information. Furthermore, the same might be true for the projection of the scattering operator onto the moment-space given by $\ints{\basis\collision{\distribution}}$. This is the so-called \emph{closure problem}. One usually prescribes some \emph{ansatz} distribution $\ansatz[\moments](\timevar,\x,\SCheight):=\ansatz(\moments(\timevar,\x),\basis(\SCheight))$ to calculate the unknown quantities in \eqref{eq:MomentSystemUnclosed1D}. Note that the dependence on the angular basis in the short-hand notation $\ansatz[\moments]$ is neglected for notational simplicity.\\

In this paper the ansatz density $\ansatz$ is reconstructed from the moments $\moments$ by minimizing the entropy-functional 
 \begin{align}
 \label{eq:entropyFunctional}
 \entropyFunctional(\distribution) = \ints{\entropy(\distribution)}
 \end{align}
 under the moment constraints
 \begin{align}
 \label{eq:MomentConstraints}
 \ints{\basis\distribution} = \moments.
 \end{align}
The kinetic entropy density $\entropy:\R\to\R$ is strictly convex and twice continuously differentiable and the minimum is simply taken over all functions $\distribution = \distribution(\SCheight)$ such that 
  $\entropyFunctional(\distribution)$ is well defined. The obtained ansatz $\ansatz = \ansatz[\moments]$, solving this constrained optimization problem, is given by
 \begin{equation}
  \ansatz[\moments] = \argmin\limits_{\distribution:\entropy(\distribution)\in\Lp{1}}\left\{\ints{\entropy(\distribution)}
  : \ints{\basis \distribution} = \moments \right\}.
 \label{eq:primal}
 \end{equation}
This problem, which must be solved over the space-time mesh, is typically solved through its strictly convex finite-dimensional dual,
 \begin{equation}
  \multipliers(\moments) := \argmin_{\tilde{\multipliers} \in \R^{\momentnumber}} \ints{\ld{\entropy}(\basis^T 
   \tilde{\multipliers})} - \moments^T \tilde{\multipliers},
 \label{eq:dual}
 \end{equation}
where $\ld{\entropy}$ is the Legendre dual of $\entropy$. The first-order necessary conditions for the multipliers $\multipliers(\moments)$ show that the solution to \eqref{eq:primal} has the form
 \begin{equation}
  \ansatz[\moments] = \ld{\entropy}' \left(\basis^T \multipliers(\moments) \right),
 \label{eq:psiME}
 \end{equation}
where $\ld{\entropy}'$ is the derivative of $\ld{\entropy}$.\\

This approach is called the \emph{minimum-entropy closure} \cite{Levermore1996}. The resulting model has many desirable properties: symmetric hyperbolicity, bounded eigenvalues of the directional flux Jacobian and the direct existence of an entropy-entropy flux pair (compare \cite{Levermore1996,Schneider2016}).\\

The kinetic entropy density $\entropy$ can be chosen according to the 
physics being modelled.
As in \cite{Levermore1996,Hauck2010}, Maxwell-Boltzmann entropy%
 \begin{align}
 \label{eq:EntropyM}
  \entropy(\distribution) = \distribution \log(\distribution) - \distribution
 \end{align}
is used, thus $\ld{\entropy}(p) = \ld{\entropy}'(p) = \exp(p)$. This entropy is used for non-interacting particles as in an ideal gas.

Substituting $\distribution$ in \eqref{eq:MomentSystemUnclosed1D} with $\ansatz[\moments]$ yields a closed system of equations for $\moments$:
\begin{align}
\label{eq:MomentSystemClosed}
\dt\moments+\dz\ints{\SCheight \basis\ansatz[\moments]} + \absorption\moments = \scattering\ints{\basis\collision{\ansatz[\moments]}}+\ints{\basis\source}.
\end{align}

For convenience, \eqref{eq:MomentSystemClosed} can be written in the form of a usual first-order hyperbolic system of balance laws
\begin{align}
\label{eq:GeneralHyperbolicSystem}
\dt\moments+\dx\Flux\left(\moments\right) = \Source\left(\moments\right),
\end{align}
where 
\begin{subequations}
\label{eq:FluxDefinitions}
\begin{align}
\Flux\left(\moments\right) &= \ints{\SCheight\basis\ansatz[\moments]}\in\R^{\momentnumber},\\
\Source\left(\moments\right) &= \scattering\ints{\basis\collision{\ansatz[\moments]}}+\ints{\basis\source}-\absorption\moments.
\end{align}
\end{subequations}

In this paper, a variant of the so-called mixed-moment basis \cite{Schneider2014,Frank07} is used. This ansatz is a combination of the \emph{full-moment} ($\basiscomp = \SCheight^\basisind$) and \emph{half-moment monomial basis} ($\basiscomp = \indicator{[-1,0]}\SCheight^\basisind$ or $\basiscomp = \indicator{[0,1]}\SCheight^\basisind$) \cite{DubKla02,DubFraKlaTho03}. The classical mixed-moment basis consists of a full zeroth moment and half moments for every higher moment. The resulting ansatz \eqref{eq:psiME} is continuous but not continuously differentiable in $\SCheight = 0$, leading to a microscopic term of the form $\ansatz[\moments](0)$ in the scattering term $\ints{\basis \LaplaceBeltramiProjection\ansatz[\moments]}$ \cite{Frank07,Schneider2014}. While this can be treated easily in one dimension, discretization problems arise in higher dimensions, where the microscopic quantity has to be replaced by an integration over a spherical arc of the unit sphere \cite{Schneider2016,Schneider2015c}.

For this reason, we modify the mixed-moment basis in such a way that the ansatz is differentiable in $\SCheight=0$, removing the microscopic quantity. In one dimension, it suffices to choose a full zeroth and first moment to obtain the desired regularity. The corresponding moments have the form
\begin{align*}
\momentcomp{\basisind} &= \ints{\SCheight^\basisind\distribution} =: \ints{\basiscomp[\basisind]\distribution}, &\basisind\in\{0,1\},\\
\momentcomp{\basisind\pm} &= \intpm{\SCheight^\basisind\distribution} =: \ints{\basiscomp[\basisind\pm]\distribution}, &\basisind\geq 2,\\
\normalizedmomentcomp{1} &= \frac{\momentcomp{1}}{\momentcomp{0}}, &\\
\normalizedmomentcomp{\basisind\pm} &= \frac{\momentcomp{\basisind\pm}}{\momentcomp{0}}, &\basisind\geq 2,
\end{align*}
where $\intp{\cdot} = \int\limits_{0}^1\cdot~d\SCheight$ and $\intm{\cdot} = \int\limits_{-1}^0\cdot~d\SCheight$ denote integration over the halfspaces. Accordingly, the angular basis has the form $\basis = \left(1,\SCheight,\indicator{[0,1]}\SCheight^2,\ldots,\indicator{[0,1]}\SCheight^\momentorder,\indicator{[-1,0]}\SCheight^2,\ldots,\indicator{[-1,0]}\SCheight^\momentorder\right)^T = \left(\basiscomp[0],\basiscomp[1],\basiscomp[2+],\ldots,\basiscomp[\momentorder+],\basiscomp[2-],\ldots,\basiscomp[\momentorder-]\right)^T$.

Using this basis, it holds that
\begin{align}
\ints{\basiscomp[0]\LaplaceBeltramiProjection \distribution} &= 0,\nonumber\\
\ints{\basiscomp[1]\LaplaceBeltramiProjection \distribution} &= -2\momentcomp{1},\label{eq:LaplaceBeltramiMM1D}\\
\ints{\basiscomp[\SHl\pm]\LaplaceBeltramiProjection \distribution} &= -\SHl(\SHl+1)\momentcomp{\SHl\pm}+\SHl(\SHl-1)\momentcomp{(\SHl-2)\pm},\qquad \SHl\in\{2,\ldots,\momentorder\}.\nonumber
\end{align}
Note that for $\SHl = 2$ and $\SHl=3$ the quantities $\momentcomp{0\pm} = \intpm{\distribution}$ and $\momentcomp{1\pm} = \intpm{\SCheight\distribution}$, respectively, appear, which have to be determined using the closure relation \eqref{eq:psiME}.

\begin{definition}
The \textbf{classical mixed-moment model} will be referred to as the $\MMN$ model, while the \textbf{differentiable mixed-moment model} will be called the $\DMMN$ model.
\end{definition}

\begin{figure}[h!]
\externaltikz{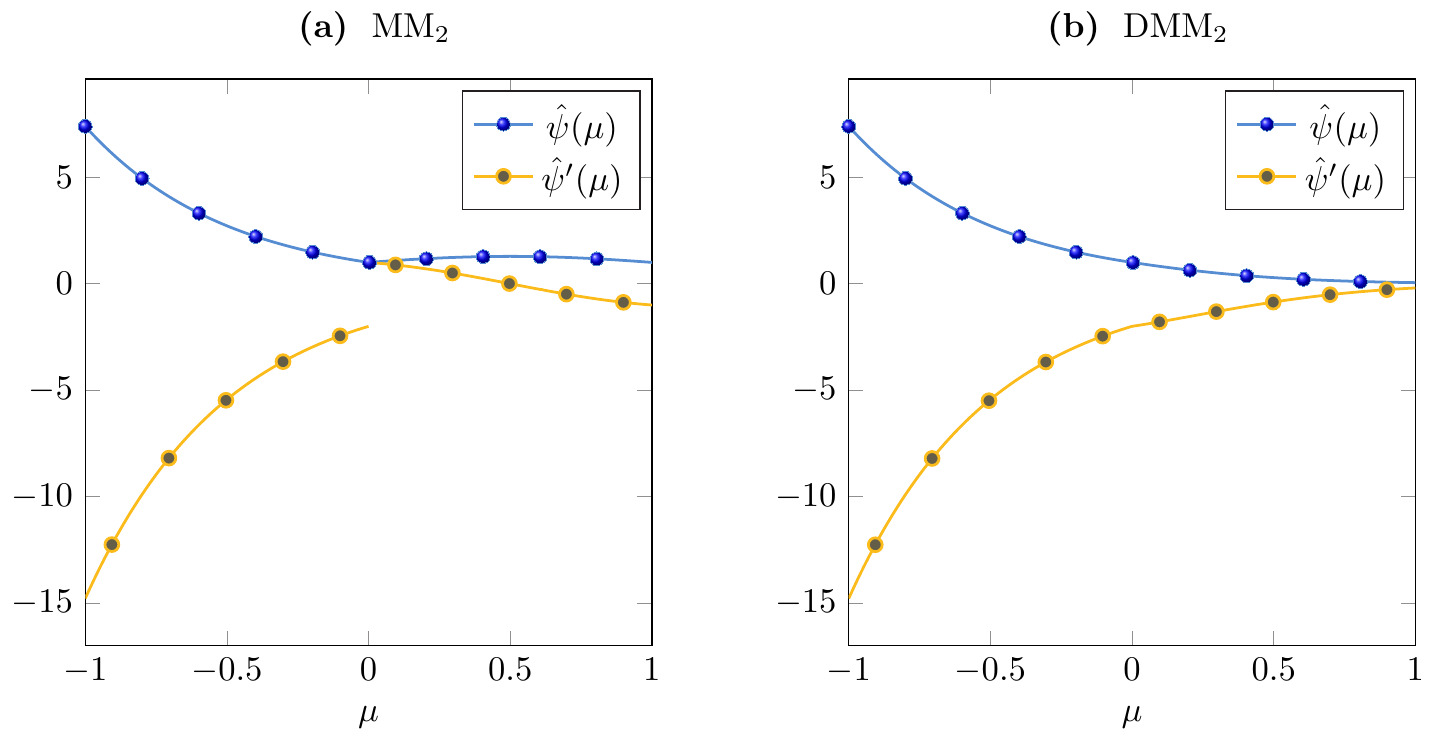}{\relinput{Images/Ansatz}}
\centering
\caption{Two ansatz functions and their derivatives for the $\MMN[2]$ and $\DMMN[2]$ model, respectively.\\
\textbf{Left}: $\ansatz(\SCheight) = \exp\left(\left(\SCheight-\SCheight^2\right)\indicator{[0,1]}-2\SCheight\indicator{[-1,0]}\right)$,\qquad
\textbf{Right}: $\ansatz(\SCheight) = \exp\left(-2\SCheight-\SCheight^2\indicator{[0,1]}\right)$
}
 \label{fig:Ansatz}
\end{figure}

\figref{fig:Ansatz} shows typical ansatz functions $\ansatz$ for the $\MMN[2]$ and $\DMMN[2]$ model. It can be seen that the $\MMN[2]$ ansatz is only continuous, while the $\DMMN[2]$ ansatz is also continuously differentiable in $\SCheight$.
\section{Realizability}

Since the underlying kinetic density to be approximated is
non-negative, a 
moment vector only makes sense physically if it can be associated with a 
non-negative distribution function. In this case the moment vector is called 
\emph{realizable}.

\begin{definition}
\label{def:RealizableSet}
The \emph{realizable set} $\RD{\basis}{}$\index{Realizability@\textbf{Realizability}!Realizable set $\RD{\basis}{}$} is 
$$
\RD{\basis}{} = \left\{\moments~:~\exists \distribution(\SCheight)\ge 0,\, \density = \ints{\distribution} > 0,
 \text{ such that } \moments =\ints{\basis\distribution} \right\}.
$$
If $\moments\in\RD{\basis}{}$, then $\moments$ is called \emph{realizable}.
Any $\distribution$ such that $\moments =\ints{\basis \distribution}$ is called a \emph{representing 
density}. 
\end{definition}
\begin{remark}
\mbox{ }
\begin{enumerate}[(a)]
\item The realizable set is a convex cone, and
\item Representing densities are not necessarily unique.
\end{enumerate}
\end{remark}

Additionally, since the entropy ansatz has the form \eqref{eq:psiME}, in the 
Maxwell-Boltzmann case, the optimization problem \eqref{eq:primal} only has a 
solution if the moment vector lies in the ansatz space
$$
 \AnsatzSpace := \left\{\ints{\basis \ansatz[\moments]}\stackrel{\eqref{eq:psiME}}{=} \ints{\basis \ld{\entropy}'\left(\basis^T\multipliers\right) }
  : \multipliers \in \R^{\momentnumber}  \right\}.
$$
In the case of a bounded angular domain, the ansatz 
space $\AnsatzSpace$ is equal to the set of realizable moment vectors 
\cite{Jun00}. Therefore, it is sufficient to focus on realizable moments only.

Unfortunately, the definition of the realizable set is not constructive, making it hard to check if a moment vector is realizable or not. Therefore, other characterizations of $\RD{\basis}{}$ are necessary.

For example, in the classical mixed-moment problem of first order, the realizable set is characterized by the inequalities \cite{Frank07,Schneider2014}
\begin{align*}
\momentcomp{1+}-\momentcomp{1-}\leq \momentcomp{0} \quand \pm \momentcomp{1\pm}\geq 0.
\end{align*}

In this paper, we want to focus on the lowest-order non-trivial model of the differentiable mixed-moment hierarchy, i.e. $\momentorder = 2$.

\begin{theorem}
The moment vector $\moments = \left(\momentcomp{0},\momentcomp{1},\momentcomp{2+},\momentcomp{2-}\right)\in\R^4$ is realizable, i.e. $\moments\in\RD{\basis}{}$, if and only if
\begin{align}
\label{eq:Realizability1D}
\momentcomp{2+} - \sqrt{\momentcomp{2-}\, \left(\momentcomp{0}-\momentcomp{2+}\right)} &\leq \momentcomp{1} \leq \sqrt{\momentcomp{2+}\, \left(\momentcomp{0}-\momentcomp{2-}\right)} - \momentcomp{2-},\\
\momentcomp{0},\momentcomp{2\pm}\geq 0. \label{eq:Realizability1Db}
\end{align}
\end{theorem}
\begin{proof}
At first, we want to show that \eqref{eq:Realizability1D} and \eqref{eq:Realizability1Db} are necessary. Assume that $\distribution\geq 0$ is arbitrary but fixed and $\moments = \ints{\basis\distribution}$. Note that $\momentcomp{0\pm}\geq \pm\momentcomp{1\pm}\geq \momentcomp{2\pm}\geq 0$ and $\momentcomp{0\pm}\momentcomp{2\pm}\geq \momentcomp{1\pm}^2$ due to the half-moment realizability conditions \cite{Schneider2014,Curto1991}. Then we have (using $\momentcomp{0-}+\momentcomp{0+} = \momentcomp{0}$) that 
\begin{align*}
\momentcomp{2-}\, \left(\momentcomp{0}-\momentcomp{2+}\right) = \momentcomp{2-}\momentcomp{0-}+ \momentcomp{2-}\left(\momentcomp{0+}-\momentcomp{2+}\right) \geq \momentcomp{1-}^2 + \momentcomp{2-}\left(\momentcomp{0+}-\momentcomp{2+}\right) \geq \momentcomp{1-}^2.
\end{align*}
Since $\momentcomp{1-}\leq 0$ it follows that 
\begin{gather*}
\sqrt{\momentcomp{2-}\, \left(\momentcomp{0}-\momentcomp{2+}\right) } \geq \abs{\momentcomp{1-}} = -\momentcomp{1-}
\quad \Longleftrightarrow\quad
-\sqrt{\momentcomp{2-}\, \left(\momentcomp{0}-\momentcomp{2+}\right) } \leq \momentcomp{1-}.
\end{gather*}
Therefore 
\begin{align*}
\momentcomp{2+} - \sqrt{\momentcomp{2-}\, \left(\momentcomp{0}-\momentcomp{2+}\right)} \leq \momentcomp{2+}+\momentcomp{1-} \leq \momentcomp{1+}+\momentcomp{1-} = \momentcomp{1}.
\end{align*}
The upper bound can be shown to be necessary in a similar way.\\
\eqref{eq:Realizability1Db} follows from the positivity of $1$ and $\SCheight^2$. We want to remark that the standard second-order full-moment realizability condition for $\momentcomp{2} = \momentcomp{2+}+\momentcomp{2-}$, namely $\momentcomp{0}\left(\momentcomp{2+}+\momentcomp{2-}\right)\geq \momentcomp{1}^2$, is implied by \eqref{eq:Realizability1D} and \eqref{eq:Realizability1Db}. 

To show that the above inequalities are also sufficient, we provide a non-negative realizing distribution with support in $[-1,1]$:
\begin{align*}
\distribution = \momentcomp{0}\left(\cfrac{\normalizedmomentcomp{1+}^2}{\normalizedmomentcomp{2+}}\cdot\dirac\left(\SCheight-\cfrac{\normalizedmomentcomp{2+}}{\normalizedmomentcomp{1+}}\right)+\cfrac{\normalizedmomentcomp{1-}^2}{\normalizedmomentcomp{2-}}\cdot\dirac\left(\SCheight-\cfrac{\normalizedmomentcomp{2-}}{\normalizedmomentcomp{1-}}\right)\right),
\end{align*}
with
\begin{align*}
\normalizedmomentcomp{1+} = \frac{\normalizedmomentcomp{2+}\, \left(\normalizedmomentcomp{1} + \normalizedmomentcomp{2-}\, \sqrt{\frac{ - {\normalizedmomentcomp{1}}^2 + \normalizedmomentcomp{2-} + \normalizedmomentcomp{2+}}{\normalizedmomentcomp{2-}\, \normalizedmomentcomp{2+}}}\right)}{\normalizedmomentcomp{2-} + \normalizedmomentcomp{2+}}\\
\normalizedmomentcomp{1-} = \frac{\normalizedmomentcomp{2-}\, \left(\normalizedmomentcomp{1} - \normalizedmomentcomp{2+}\, \sqrt{\frac{ - {\normalizedmomentcomp{1}}^2 + \normalizedmomentcomp{2-} + \normalizedmomentcomp{2+}}{\normalizedmomentcomp{2-}\, \normalizedmomentcomp{2+}}}\right)}{\normalizedmomentcomp{2-} + \normalizedmomentcomp{2+}}
\end{align*}
and $\normalizedmomentcomp{1+}=\normalizedmomentcomp{1-} = 0$ if $\normalizedmomentcomp{2+}=\normalizedmomentcomp{2-} = 0$. In this case, $\distribution = \momentcomp{0}\dirac(\SCheight)$ (due to the quadratic term the second moment vanishes faster than the first moment so we have $\frac{\normalizedmomentcomp{2\pm}}{\normalizedmomentcomp{1\pm}}\to 0$ and $\frac{\normalizedmomentcomp{1\pm}^2}{\normalizedmomentcomp{2\pm}}\to \frac{\normalizedmomentcomp{2\mp}-\normalizedmomentcomp{1}^2}{\normalizedmomentcomp{2\mp}}$). It is simple to check that $\normalizedmomentcomp{1+}+\normalizedmomentcomp{1-} = \normalizedmomentcomp{1}$ and $\cfrac{\normalizedmomentcomp{1+}^2}{\normalizedmomentcomp{2+}}+\cfrac{\normalizedmomentcomp{1-}^2}{\normalizedmomentcomp{2-}} = 1$, i.e. all moments are correctly represented. It remains to show that under \eqref{eq:Realizability1D} we have that $\cfrac{\normalizedmomentcomp{2+}}{\normalizedmomentcomp{1+}}\in[0,1]$ and $\cfrac{\normalizedmomentcomp{2-}}{\normalizedmomentcomp{1-}}\in[-1,0]$, i.e. 
\begin{align*}
0 \leq \normalizedmomentcomp{2\pm} \leq \pm\normalizedmomentcomp{1\pm}.
\end{align*}
This corresponds to the standard half-moment realizability conditions of second order. 
We first note that \eqref{eq:Realizability1D} and \eqref{eq:Realizability1Db} imply that $\normalizedmomentcomp{2\pm}\in [0,1]$ since otherwise the bounds become complex. Second, we have that $\normalizedmomentcomp{2+} - \sqrt{\normalizedmomentcomp{2-}\, \left(1-\normalizedmomentcomp{2+}\right)} = \sqrt{\normalizedmomentcomp{2+}\, \left(1-\normalizedmomentcomp{2-}\right)} - \normalizedmomentcomp{2-}$ if and only if $\normalizedmomentcomp{2+} = 1-\normalizedmomentcomp{2-}$ or $\normalizedmomentcomp{2+} = \normalizedmomentcomp{2-} = 0$, implying the classical full-moment realizability conditions $\normalizedmomentcomp{1}\in[-1,1]$ and $\normalizedmomentcomp{2}\leq 1$.

We start the investigation at the different parts of the realizability boundary. \\
Let $\normalizedmomentcomp{1} = \sqrt{\normalizedmomentcomp{2+}\, \left(1-\normalizedmomentcomp{2-}\right)} - \normalizedmomentcomp{2-}$. Plugging this into the definition of $\normalizedmomentcomp{1+}$ we get that, after some elementary transformations,
\begin{align*}
\normalizedmomentcomp{1+} &\stackrel{\phantom{1\geq\normalizedmomentcomp{2+}+\normalizedmomentcomp{2-}}}{\geq} \cfrac{\normalizedmomentcomp{2+}}{\normalizedmomentcomp{2+}+\normalizedmomentcomp{2-}}\left(\sqrt{1-\normalizedmomentcomp{2-}}\left(\sqrt{\normalizedmomentcomp{2+}}+\cfrac{\normalizedmomentcomp{2-}}{\sqrt{\normalizedmomentcomp{2+}}}\right)\right)\\
&\stackrel{1\geq\normalizedmomentcomp{2+}+\normalizedmomentcomp{2-}}{\geq} \cfrac{\normalizedmomentcomp{2+}}{\normalizedmomentcomp{2+}+\normalizedmomentcomp{2-}}\left(\sqrt{\normalizedmomentcomp{2+}}\left(\sqrt{\normalizedmomentcomp{2+}}+\cfrac{\normalizedmomentcomp{2-}}{\sqrt{\normalizedmomentcomp{2+}}}\right)\right) = \normalizedmomentcomp{2+}.
\end{align*}
Similarly, we obtain $-\normalizedmomentcomp{1-}\geq \normalizedmomentcomp{2-}$ and the same in the case $\normalizedmomentcomp{1} = \normalizedmomentcomp{2+} - \sqrt{\normalizedmomentcomp{2-}\, \left(1-\normalizedmomentcomp{2+}\right)}$.

%
%

Since the realizable set is always convex, the argumentation must also hold in the interior of the above set.

\end{proof}

\begin{figure}[h!]

\begin{center}
\includemedia[
label=RealizableSet,
activate=pageopen,
3Dtoolbar,
3Dviews=Images/RealizableSet.vws,
]
{\externaltikz{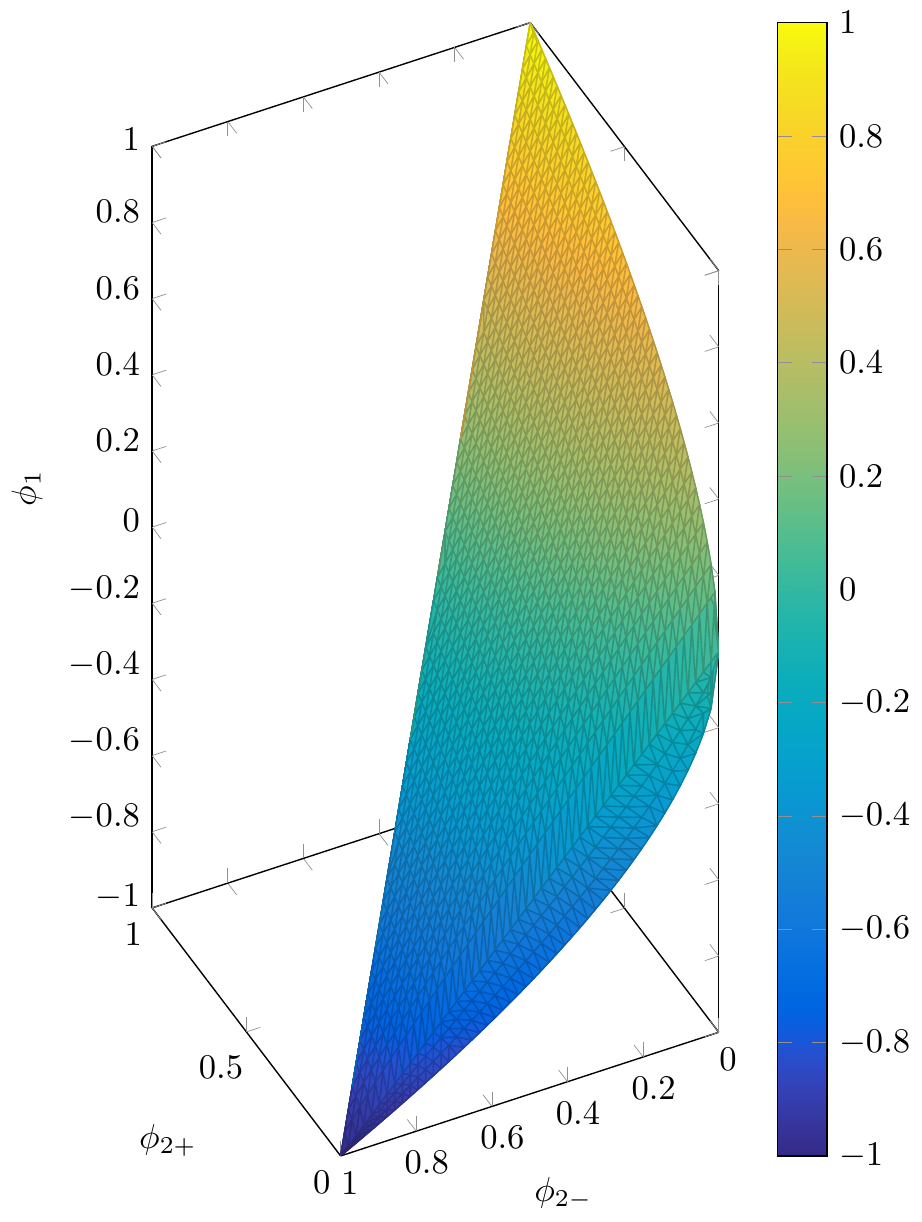}{\relinput{Images/Realizability}}}{Images/RealizableSet.u3d}
\end{center}
\caption{The normalized realizable set for the differentiable mixed-moment basis of order $\momentorder=2$.\\
\textbf{Online version}: Press to activate 3D view ($\x$-axis (red): $\normalizedmomentcomp{2+}$, $\y$-axis (green): $\normalizedmomentcomp{2-}$, $z$-axis (blue): $\normalizedmomentcomp{1}$)}
 \label{fig:Realizability}
\end{figure}
The \emph{normalized realizable set}
\begin{align*}
\RDone{\basis} = \left\{\moments\in \RD{\basis}{}~|~\density = 1 \right\}
\end{align*}
of the $\DMMN[2]$ model, defined by \eqref{eq:Realizability1D} and \eqref{eq:Realizability1Db}, is shown in \figref{fig:Realizability}.

\begin{remark}
\eqref{eq:Realizability1D} gives a surprising insight into realizability of mixed-moment models. While the realizable set for full-moment and classical mixed-moment models can be characterized by inequalities with rational functions of the moments, the differential mixed-moment model requires non-linearities. This implies that it might be impossible to transfer the general mixed-moment structure (which uses a linearity argument) shown in \cite{Schneider2014} to the differentiable case. 
\end{remark}

\section{Eigenstructure of the $\DMMN[2]$ model}
It is well known that the moment system \eqref{eq:MomentSystemClosed} admits desirable properties like symmetric hyperbolicity, boundedness of the characteristic velocities (eigenvalues of the flux Jacobian) and the existence of an entropy-entropy flux pair \cite{Levermore1998,Schneider2015a,Schneider2016,Levermore1996}. Furthermore, the eigenvalues only depend on the normalized moments $\normalizedmoments$.

If we define $\optJacobian(\multipliers) := \ints{\SCheight \basis \basis^T \ld{\entropy}''(\basis^T 
\multipliers)}$ and $\optHessian(\multipliers):= \ints{\basis \basis^T \ld{\entropy}''(\basis^T 
\multipliers)}$, the flux Jacobian of \eqref{eq:GeneralHyperbolicSystem} has the form \cite{Schneider2015a,Schneider2016,Levermore1996}
\begin{equation}
 \frac{\partial \Flux(\moments)}{\partial \moments} = \optJacobian(\multipliers(\moments))
 \frac{\partial \multipliers(\moments)}{\partial \moments}
  = \optJacobian(\multipliers(\moments)) \optHessian(\multipliers(\moments))^{-1}.
\label{eq:fluxJacobian}
\end{equation}

In the special case of the $\DMMN[2]$ model, the flux Jacobian is given by
\begin{equation}
\label{eq:fluxJacobianDMM2}
 \frac{\partial \Flux(\moments)}{\partial \moments} = \begin{pmatrix}
 0 & 1 & 0 & 0\\
 0 & 0 & 1 & 1\\
 \multicolumn{4}{c}{\frac{\partial \momentcomp{3+}}{\partial \moments}}\\
 \multicolumn{4}{c}{\frac{\partial \momentcomp{3-}}{\partial \moments}}
  \end{pmatrix},
\end{equation}
where $\momentcomp{3\pm} = \intpm{\SCheight^3\ansatz[\moments]}$ is obtained via the closure relation. 

The four eigenvalues $\eigenvalue_1\leq \eigenvalue_2\leq \eigenvalue_3\leq \eigenvalue_4$ of \eqref{eq:fluxJacobianDMM2}, which have been obtained numerically, are shown in Figures~\ref{fig:Eigenvalues} to \ref{fig:Eigenvalues4}.
\begin{figure}
\centering
\externaltikz{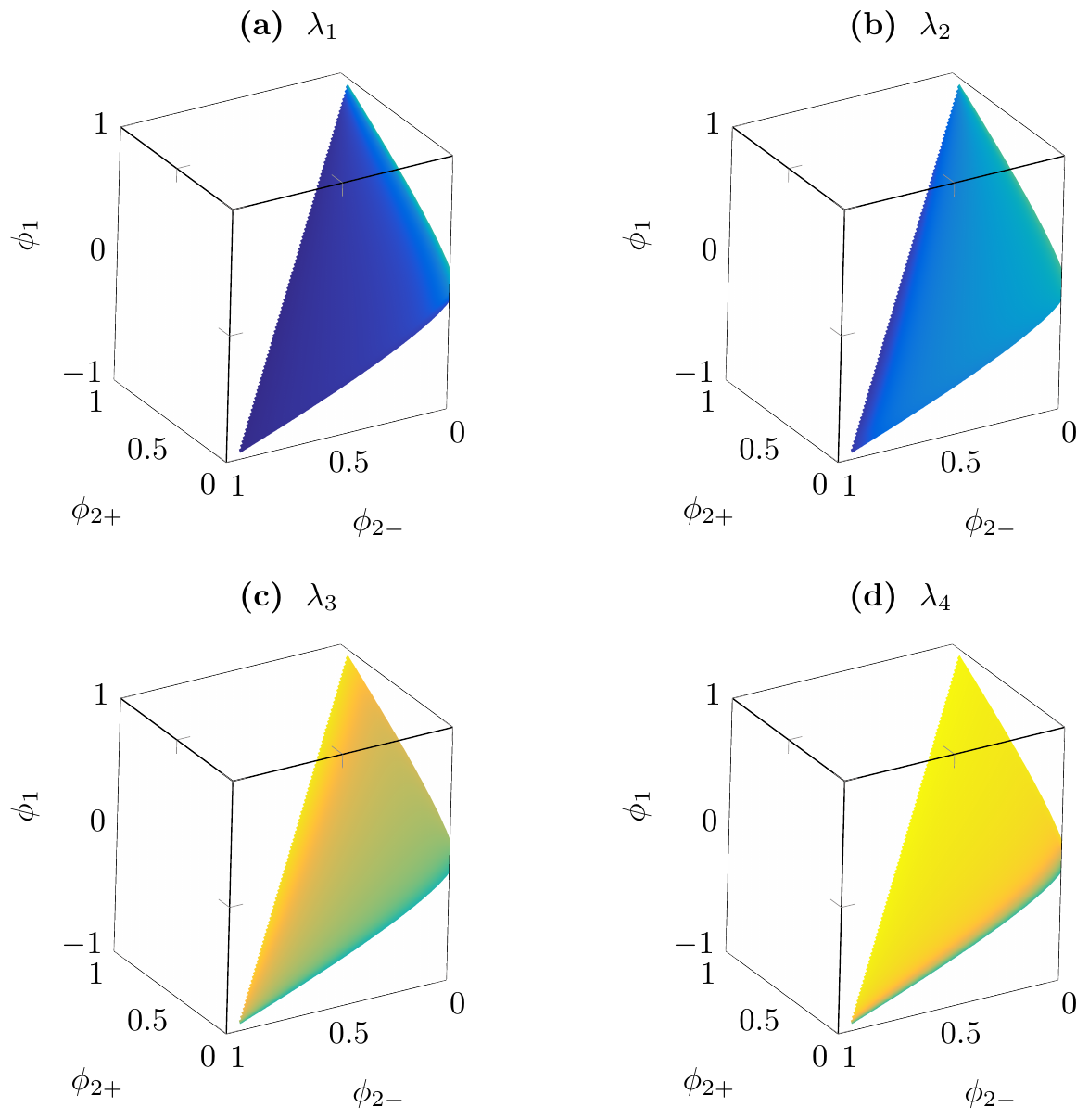}{\relinput{Images/Eigenvalues}}
\hspace{-0.0cm}
\externaltikz{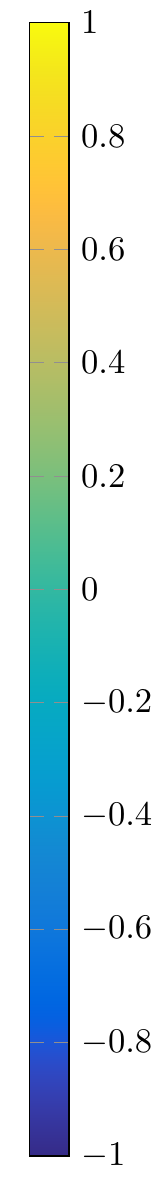}{\relinput{Images/colorbar}}
\caption{Eigenvalues of the $\DMMN[2]$ flux Jacobian $\frac{\partial \Flux(\moments)}{\partial \moments}$ along $\normalizedmomentcomp{1} = \frac12\left(\normalizedmomentcomp{2+} - \sqrt{\normalizedmomentcomp{2-}\, \left(1-\normalizedmomentcomp{2+}\right)} +\sqrt{\normalizedmomentcomp{2+}\, \left(1-\normalizedmomentcomp{2-}\right)} - \normalizedmomentcomp{2-}\right)$.}
\label{fig:Eigenvalues}
\end{figure}
In \figref{fig:Eigenvalues}, the eigenvalues are shown along the cut $$\normalizedmomentcomp{1} = \frac12\left(\normalizedmomentcomp{2+} - \sqrt{\normalizedmomentcomp{2-}\, \left(1-\normalizedmomentcomp{2+}\right)} +\sqrt{\normalizedmomentcomp{2+}\, \left(1-\normalizedmomentcomp{2-}\right)} - \normalizedmomentcomp{2-}\right),$$ which is exactly the mean of the upper and lower bound on $\normalizedmomentcomp{1}$. 

It can be seen that the eigenvalues are discontinuous in the degenerate corners of the realizable set (e.g. $\eigenvalue_4$ at $\normalizedmomentcomp{2-}=1=-\normalizedmomentcomp{1}$, $\normalizedmomentcomp{2+}=0$). This property exists also for the classical mixed-moment $\MMN[1]$ or the $\MN[2]$ model \cite{Schneider2016}.

\begin{figure}
\centering
\settikzlabel{fig:Eigenvaluesmax}
\settikzlabel{fig:Eigenvaluesmin}
\externaltikz{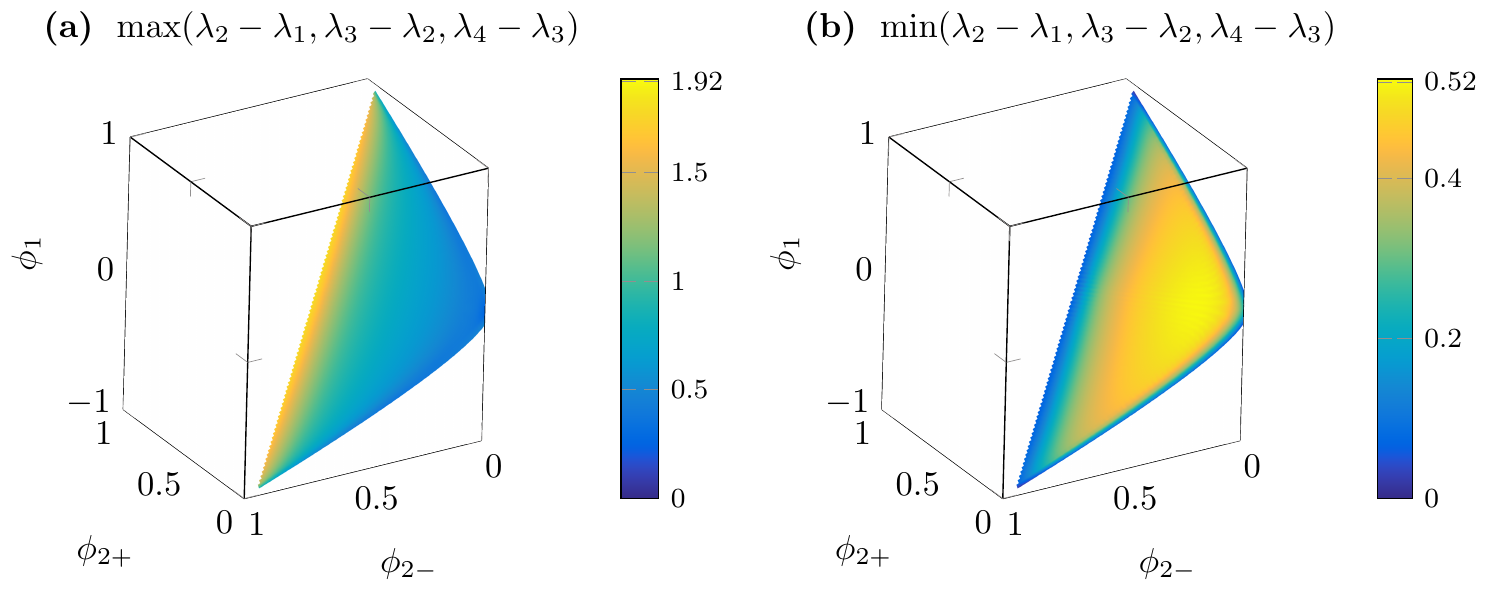}{\relinput{Images/Eigenvalues2}}
\caption{Minimal and maximal distance between adjacent eigenvalues of the $\DMMN[2]$ flux Jacobian $\frac{\partial \Flux(\moments)}{\partial \moments}$ along the cut $\normalizedmomentcomp{1} = \frac12\left(\normalizedmomentcomp{2+} - \sqrt{\normalizedmomentcomp{2-}\, \left(1-\normalizedmomentcomp{2+}\right)} +\sqrt{\normalizedmomentcomp{2+}\, \left(1-\normalizedmomentcomp{2-}\right)} - \normalizedmomentcomp{2-}\right)$.}
\label{fig:Eigenvalues2}
\end{figure}

We investigate the hyperbolicity of the moment system in \figref{fig:Eigenvalues2} by comparing the distances of adjacent eigenvalues. \figref{fig:Eigenvaluesmax} shows that all eigenvalues coincide only if $\normalizedmomentcomp{1}=\normalizedmomentcomp{2+}=\normalizedmomentcomp{2-}=0$. Otherwise, at least two eigenvalues differ from each other. 

The results in \figref{fig:Eigenvaluesmin} propose the strict hyperbolicity of the moment system in the interior of the realizable set (since all eigenvalues differ). However, at the realizability boundary (e.g. $\normalizedmomentcomp{2+}+\normalizedmomentcomp{2-}=1$) at least two eigenvalues coincide.

\begin{figure}

\begin{center}
\includemedia[
label=eigenvalues,
3Dtoolbar,
activate=pageopen,
3Dviews=Images/eigenvalues.vws,
3Drender=Illustration,
]{\externaltikz{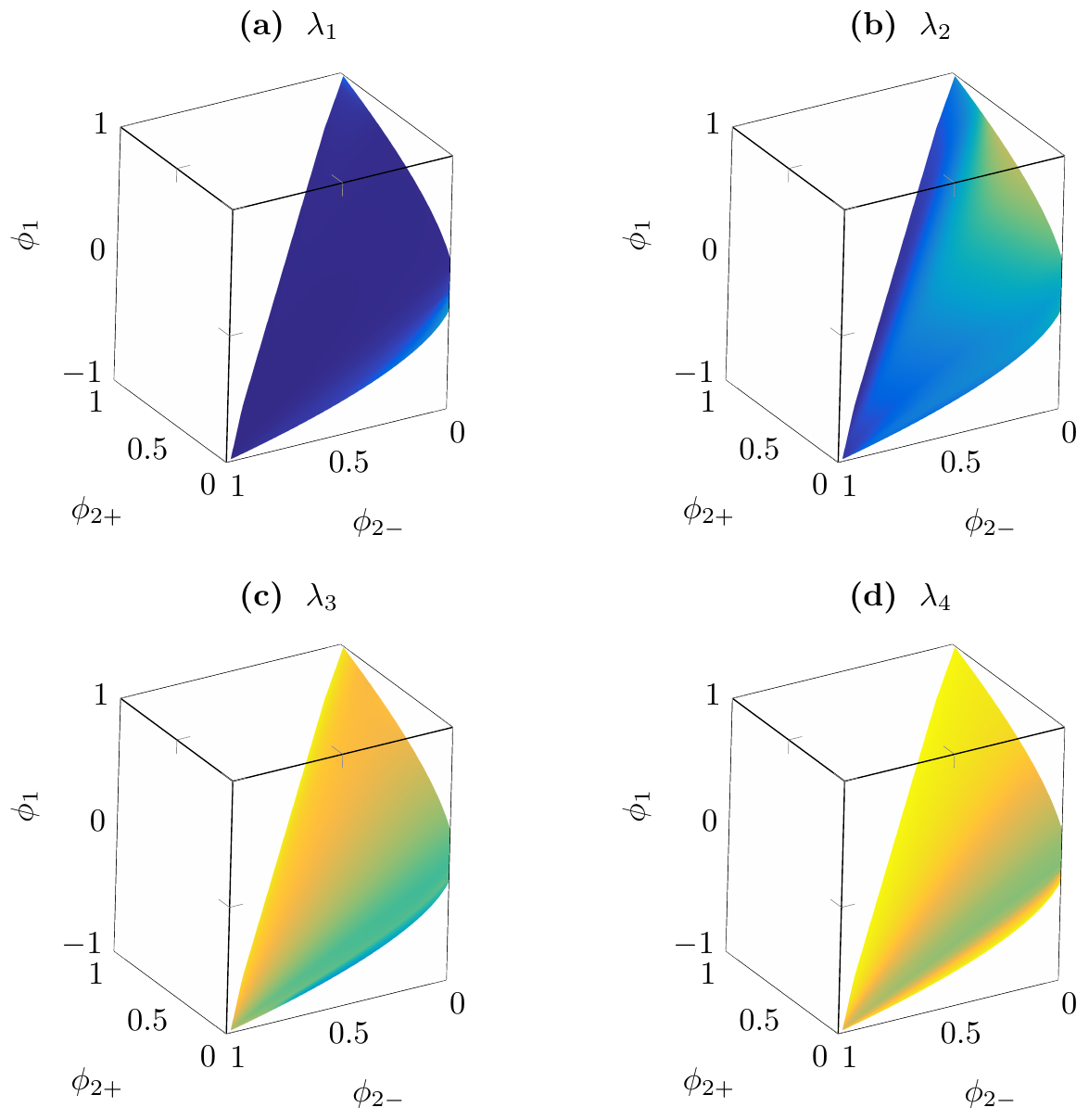}{\relinput{Images/Eigenvalues3}}}{Images/eigenvalues.u3d}\externaltikz{CB}{\relinput{Images/colorbar}}\\
\mediabutton[3Dgotoview=eigenvalues:3]{\fbox{$\eigenvalue_1$} }
\mediabutton[3Dgotoview=eigenvalues:2]{\fbox{$\eigenvalue_2$} }
\mediabutton[3Dgotoview=eigenvalues:1]{\fbox{$\eigenvalue_3$} }
\mediabutton[3Dgotoview=eigenvalues:0]{\fbox{$\eigenvalue_4$} }
\end{center}

\caption{Eigenvalues at $5\%$ regularized boundary moments.\\
\textbf{Online version}: Press to activate 3D view ($\x$-axis (red): $\normalizedmomentcomp{2+}$, $\y$-axis (green): $\normalizedmomentcomp{2-}$, $z$-axis (blue): $\normalizedmomentcomp{1}$)}
\label{fig:Eigenvalues3}
\end{figure}

Since the optimization problem \eqref{eq:dual} is ill-conditioned close to the realizability boundary, the calculation of the multipliers $\multipliers$ at this part of the realizable set is error-prone or impossible, resulting in meaningless pictures. We therefore investigate \textbf{isotropically-regularized moments}
\begin{align*}
\normalizedmoments[\regularizationParameter] = (1-\regularizationParameter)\normalizedmoments+r\normalizedisotropicmoment,
\end{align*}
where an increase of the regularization parameter $\regularizationParameter\in[0,1]$ moves the original moment vector $\normalizedmoments$ towards the isotropic moment vector (in case of the $\DMMN[2]$ model: $\normalizedisotropicmoment = \left(0,\frac16,\frac16\right)^T$). \figref{fig:Eigenvalues3} shows the eigenvalues for $\regularizationParameter = 0.05$ and $\normalizedmoments\in\dRDone{\basis}$.

Similar to \figref{fig:Eigenvaluesmin}, \figref{fig:Eigenvalues4} shows the minimal distance of the $5\%$ regularized boundary moments. It is visible that the minimal distance is attained at $\normalizedmomentcomp{2+}+\normalizedmomentcomp{2-}=1$, which indicates that on this part of the boundary the moment system is only weakly hyperbolic.
\begin{figure}

\begin{center}
\includemedia[
label=eigenvaluesdiff,
3Dtoolbar,
activate=pageopen,
3Dviews=Images/eigenvaluesdiff.vws,
3Drender=Illustration,
]{\resizebox{1.5\figurewidth}{!}{\externaltikz{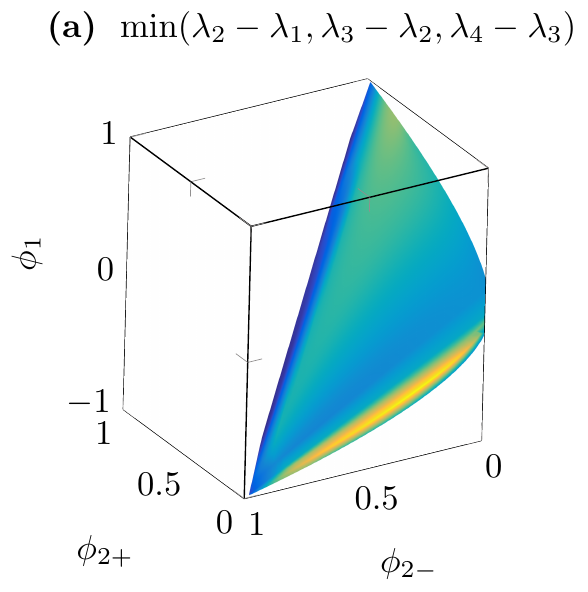}{\relinput{Images/Eigenvalues4}}}}{Images/eigenvaluesdiff.u3d}\externaltikz{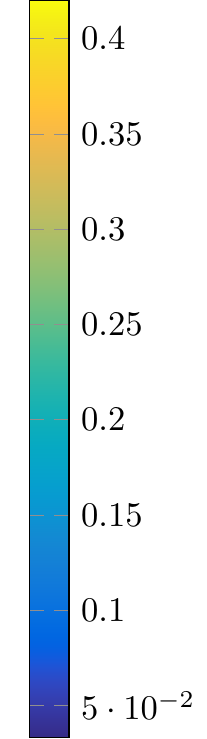}{\relinput{Images/colorbardiff}}
\end{center}

\caption{Minimal eigenvalue distance for $5\%$ regularized boundary moments.\\
\textbf{Online version}: Press to activate 3D view ($\x$-axis (red): $\normalizedmomentcomp{2+}$, $\y$-axis (green): $\normalizedmomentcomp{2-}$, $z$-axis (blue): $\normalizedmomentcomp{1}$)}
\label{fig:Eigenvalues4}
\end{figure}

However, an analytical investigation of the eigenvalues in the limit cases $\normalizedmomentcomp{1} = \normalizedmomentcomp{2+} - \sqrt{\normalizedmomentcomp{2-}\, \left(1-\normalizedmomentcomp{2+}\right)}$ and $\normalizedmomentcomp{1} = \sqrt{\normalizedmomentcomp{2+}\, \left(1-\normalizedmomentcomp{2-}\right)} - \normalizedmomentcomp{2-}$ is still open.

\section{Numerical experiments}
We use the first-order, realizability-preserving, implicit-explicit kinetic scheme derived in \cite{Schneider2016b}. All results are computed on a grid with $1000$ points. The reference solution is given by the $\PN[99]$ model \cite{Lewis-Miller-1984}. 

\subsection{Plane source}
\label{sec:Planesource}
In this test case an isotropic distribution with all mass concentrated in the middle of an infinite domain $\z \in
(-\infty, \infty)$ is defined as initial condition, i.e.
\begin{align*}
 \distributiontzero(\z, \SCheight) = \distributionvacuum + \delta(\z),
\end{align*}
where the small parameter $\distributionvacuum = 0.5 \times 10^{-8}$ is used to
approximate a vacuum.
In practice, a bounded domain must be used which is large
enough that the boundary should have only negligible effects on the
solution. For the final time $\tf = 1$, the domain is set to $\Domain = [-1.2, 1.2]$ (recall that for all presented models the maximal speed of propagation is bounded in absolute value by one).

At the boundary the vacuum approximation
\begin{align*}
 \distributionboundary(\timevar,\zL,\SCheight) \equiv \distributionvacuum \quand
 \distributionboundary(\timevar,\zR,\SCheight) \equiv \distributionvacuum
\end{align*}
 is used again. Furthermore, the physical coefficients are set to $\scattering \equiv 1$, $\absorption \equiv 0$ and $\source \equiv 0$.

All solutions are computed with an even number of cells, so the initial Dirac delta lies on a cell boundary.
Therefore it is approximated by splitting it into the cells immediately to the left and right. In \figref{fig:Planesource}, only positive $\z$ are shown since the solutions are always symmetric around $\z = 0$.\\

The figure shows the solution of the $\DMMN[2]$ model in comparison to some mixed-moment $\MMN$ and full-moment $\MN$ models with a similar number of degrees of freedom (i.e. the number of moments $\momentnumber$). 

Observe that the difference between $\MMN[2]$ and $\DMMN[2]$ is negligible\footnote{This is no longer true if the isotropic scattering operator $\collision{\distribution} = \distribution - \frac12\int\limits_{-1}^1\distribution(\SCheight')~d\SCheight'$ is used.}. Although the $\DMMN[2]$ model is exactly between $\MMN[1]$ and $\MMN[2]$ (regarding degrees of freedom), its solution is much closer to those of the $\MMN[2]$ model.

Doing the same comparison with the $\MN$ models shows that the $\DMMN[2]$ model is closer to the $\MN[2]$ than to the $\MN[3]$ model (while all three models differ insignificantly from the reference solution\footnote{This results from the quadratic dependence of the Laplace-Beltrami eigenvalues with respect to the moment order $\momentorder$.}).

\begin{figure}[h]
\externaltikz{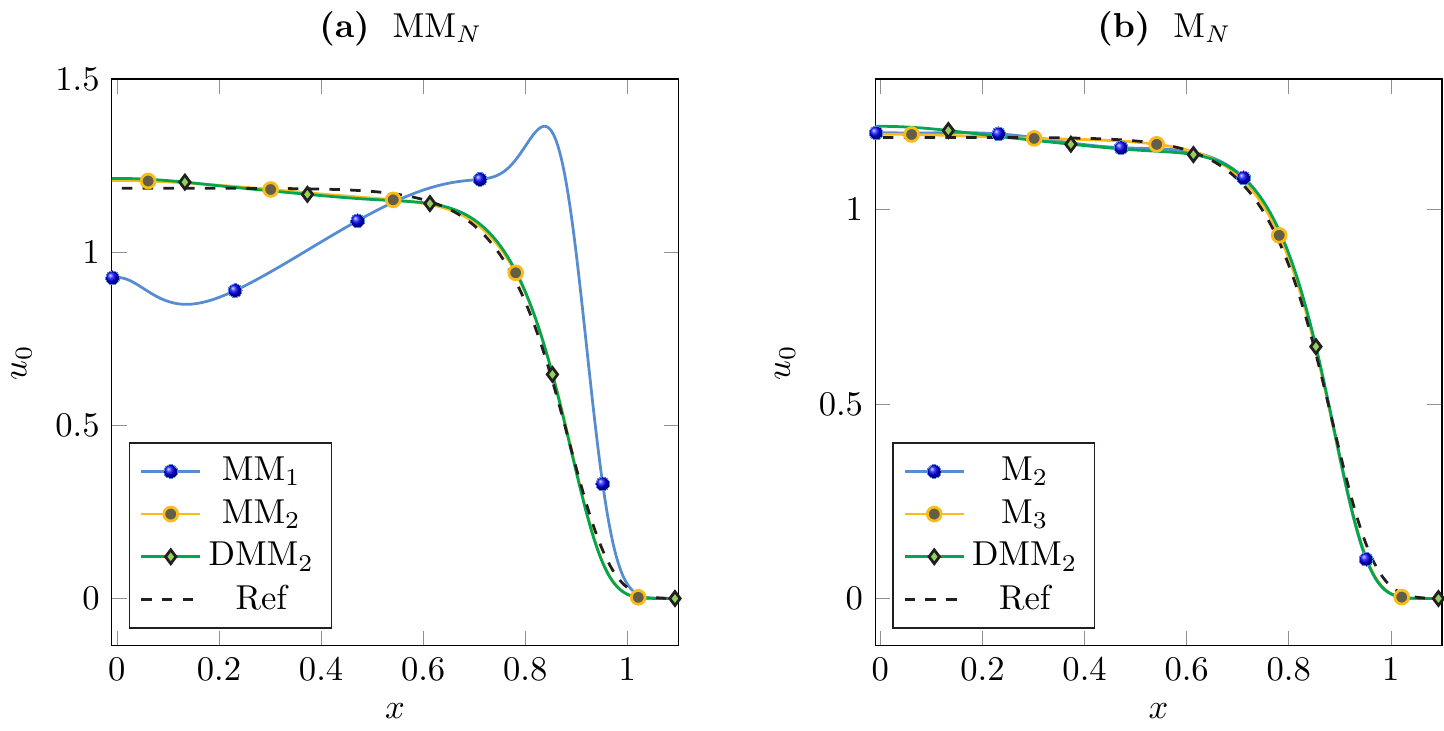}{\relinput{Images/PlanesourceLBCutsMixedMoments}}
 \centering
\caption{Results for the plane-source test at the final time $\tf = 1$.}
 \label{fig:Planesource}
\end{figure}

\subsection{Source beam}
We present a discontinuous version of the source-beam problem from
\cite{Hauck2013}, as in \cite{Schneider2015a,Schneider2015b}.
The spatial domain is $\Domain = [0,3]$, and
\begin{gather*}
 \absorption(\z) = \begin{cases}
   1 & \text{ if } \z\leq 2,\\
   0 & \text{ else},
  \end{cases} \quad
 \scattering(\z) = \begin{cases}
   0 & \text{ if } \z\leq 1,\\
   2 & \text{ if } 1<\z\leq 2,\\
   10 & \text{ else},
  \end{cases} \quad
 \source(\z) = \begin{cases}
   \frac12 & \text{ if } 1\leq \z\leq 1.5,\\
   0 & \text{ else},
  \end{cases}
\end{gather*}
with initial and boundary conditions
\begin{gather*}
 \distributiontzero(\z, \SCheight) \equiv \distributionvacuum, \\
 \distributionboundary(\timevar,\zL,\SCheight) = \cfrac{e^{-10^5(\SCheight-1)^2}}{\ints{e^{-10^5(\SCheight-1)^2}}}
 \quand
 \distributionboundary(\timevar,\zR,\SCheight) \equiv \distributionvacuum.
\end{gather*}
The final time is $\tf = 2.5$. As above, the results for $\MN$, $\MMN$ and $\DMMN$ models are shown in \figref{fig:SourceBeam}.

\begin{figure}[h]
\externaltikz{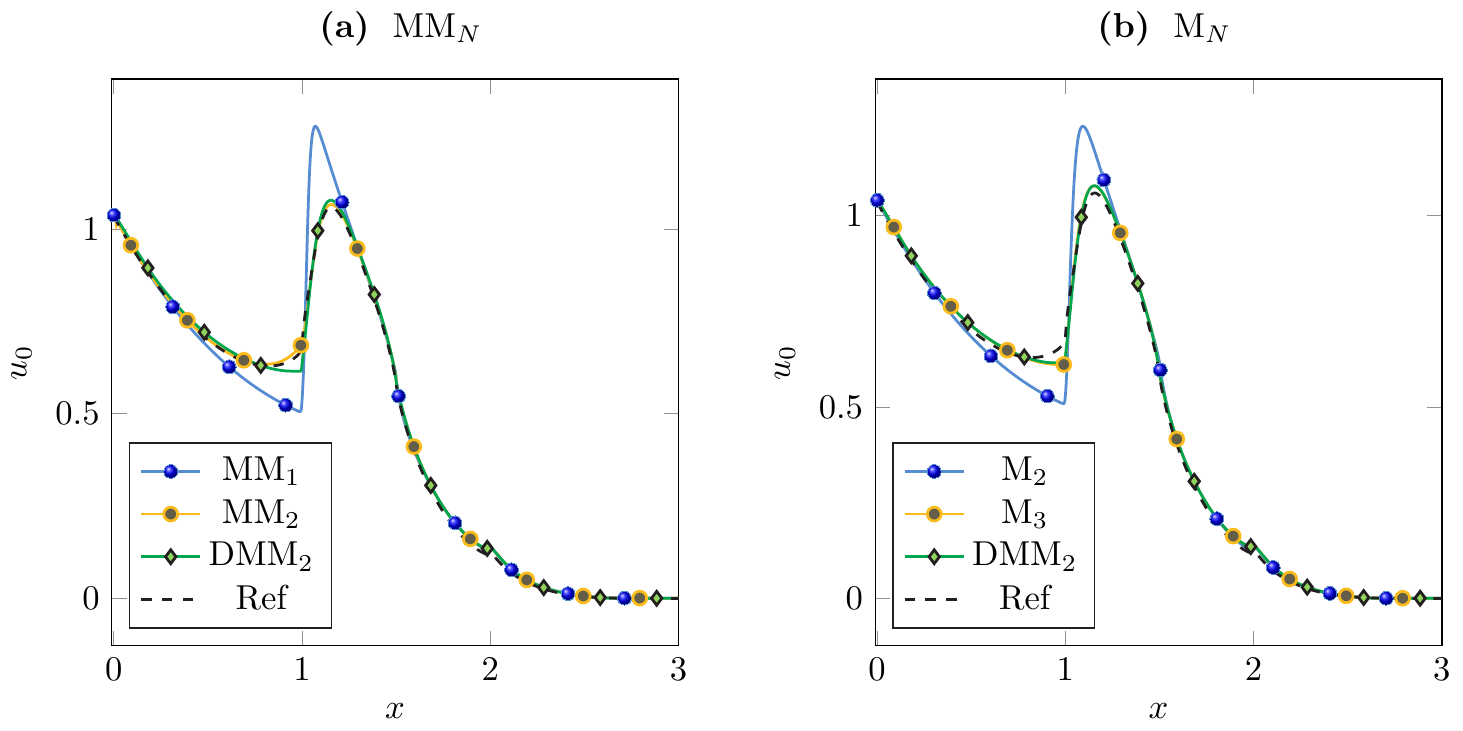}{\relinput{Images/SourceBeamLBCutsMixedMoments}}
 \centering
\caption{Results for the source-beam test at the final time $\tf = 2.5$.}
 \label{fig:SourceBeam}
\end{figure}

As has been remarked in \cite{Schneider2016}, the $\MMN[1]$ and the $\MN[2]$ model coincide well in this situation. Surprisingly, a similar statement is valid for the $\DMMN[2]$ and $\MN[3]$ model.
As before, the $\DMMN[2]$ model behaves qualitatively above the level of the $\MMN[1]$ and $\MN[2]$ model and similarly or slightly below those of the $\MMN[2]$ model (which has the highest number of degrees of freedom).

\section{Conclusions and outlook}
We have derived the $\DMMN$ model and its associated realizability domain $\RD{\basis}{}$ for $\momentorder=2$. Numerical results suggest that, despite having one degree of freedom less, the $\DMMN[2]$ model performs comparable to the $\MMN[2]$ model. The key advantage of this class of moment models is that in the approximation of the Laplace-Beltrami operator only macroscopic quantities occur, whereas microscopic terms are present in the classical mixed-moment model. While this appears to have no significant impact in one dimension, where the position of the microscopic term is well-located, a more stable numerical approximation can be expected in two or three dimensions.

Future work should include the derivation of realizability theory for moment-orders $\momentorder\geq 3$, to gain more insight into the arising non-linearities in this modified problem. Furthermore, the $\DMMN$ should be investigated in higher dimensions, especially in the context of the Fokker-Planck operator. The results in \cite{Schneider2015c,Schneider2016} indicate that mixed moments are hardly applicable in this framework due to the difficulty in the discretization of the Laplace-Beltrami operator. This should be avoidable using the differentiable basis functions. 
Finally, Kershaw closures \cite{Ker76,Schneider2016a,Schneider2015,Schneider2014} should be investigated to improve the efficiency of the $\DMMN$ model by avoiding the need to solve the moment system \eqref{eq:MomentConstraints}.

\section*{Acknowledgements}
The conversion from Matlab data to the included u3d data has been obtained using the Matlab function \emph{fig2u3d} written by Ioannis Filippidis \cite{Filippidis2015}.

\bibliographystyle{siam}
\bibliography{bibliography}

\begin{thebibliography}{10}

\bibitem{AllHau12}
{\sc G.~W. Alldredge, C.~D. Hauck, and A.~L. Tits}, {\em {High-Order
  Entropy-Based Closures for Linear Transport in Slab Geometry II: A
  Computational Study of the Optimization Problem}}, SIAM Journal on Scientific
  Computing, 34 (2012), pp.~B361--B391.

\bibitem{Schneider2015a}
{\sc G.~W. Alldredge and F.~Schneider}, {\em {A realizability-preserving
  discontinuous Galerkin scheme for entropy-based moment closures for linear
  kinetic equations in one space dimension}}, Journal of Computational Physics,
  295 (2015), pp.~665--684.

\bibitem{Boltzmann1872}
{\sc L.~Boltzmann}, {\em {Weitere Studien {\"{u}}ber das
  W{\"{a}}rmegleichgewicht unter Gasmolekulen}}, Wien. Ber., 66 (1872),
  pp.~275--370.

\bibitem{BruHol01}
{\sc T.~A. Brunner and J.~P. Holloway}, {\em {One-dimensional Riemann solvers
  and the maximum entropy closure}}, Journal of Quantitative Spectroscopy and
  Radiative Transfer, 69 (2001), pp.~543--566.

\bibitem{Brunner2005b}
\leavevmode\vrule height 2pt depth -1.6pt width 23pt, {\em {Two-dimensional
  time dependent Riemann solvers for neutron transport}}, Journal of
  Computational Physics, 210 (2005), pp.~386--399.

\bibitem{cercignani2012boltzmann}
{\sc C.~Cercignani}, {\em {The Boltzmann Equation and Its Applications}},
  Applied Mathematical Sciences, Springer New York, 2012.

\bibitem{Chalub2004}
{\sc F.~Chalub and P.~Markowich}, {\em {Kinetic models for chemotaxis and their
  drift-diffusion limits}}, Springer Vienna, Vienna, 2004.

\bibitem{Curto1991}
{\sc R.~Curto and L.~Fialkow}, {\em {Recursiveness, positivity, and truncated
  moment problems}}, Houston J. Math, 17 (1991), pp.~603--636.

\bibitem{DubFeu99}
{\sc B.~Dubroca and J.-L. Feugeas}, {\em {Entropic Moment Closure Hierarchy for
  the Radiative Transfer Equation}}, C. R. Acad. Sci. Paris Ser. I, 329 (1999),
  pp.~915--920.

\bibitem{DubFraKlaTho03}
{\sc B.~Dubroca, M.~Frank, A.~Klar, and G.~Th{\"{o}}mmes}, {\em {Half space
  moment approximation to the radiative heat transfer equations}}, ZAMM -
  Journal of Applied Mathematics and Mechanics / Zeitschrift f{\"{u}}r
  Angewandte Mathematik und Mechanik, 83 (2003), pp.~853--858.

\bibitem{DubKla02}
{\sc B.~Dubroca and A.~Klar}, {\em {Half-Moment Closure for Radiative Transfer
  Equations}}, Journal of Computational Physics, 180 (2002), pp.~584--596.

\bibitem{Eddington}
{\sc A.~S. Eddington}, {\em {The Internal Constitution of the Stars}}, Dover,
  1926.

\bibitem{Filippidis2015}
{\sc I.~Filippidis}, {\em fig2u3d},
  https://de.mathworks.com/matlabcentral/fileexchange/37640-export-figure-to-3d-interactive-pdf,
   (2015).

\bibitem{Frank2006}
{\sc M.~Frank, B.~Dubroca, and A.~Klar}, {\em {Partial moment entropy
  approximation to radiative heat transfer}}, Journal of Computational Physics,
  218 (2006), pp.~1--18.

\bibitem{Hauck2013}
{\sc M.~Frank, C.~Hauck, and E.~Olbrant}, {\em {Perturbed, entropy-based
  closure for radiative transfer}}, Kinetic and Related Models, 6 (2013),
  pp.~557--587.

\bibitem{Frank07}
{\sc M.~Frank, H.~Hensel, and A.~Klar}, {\em {A fast and accurate moment method
  for the Fokker-Planck equation and applications to electron radiotherapy}},
  SIAM Journal on Applied Mathematics, 67 (2007), pp.~582--603.

\bibitem{Gel61}
{\sc E.~M. Gelbard}, {\em {Simplified spherical harmonics equations and their
  use in shielding problems}}, Tech. Rep. WAPD-T-1182, Bettis Atomic Power
  Laboratory, 1961.

\bibitem{Hadeler2000}
{\sc K.~P. Hadeler}, {\em {Reaction transport equations in biological
  modeling}}, in Mathematical and Computer Modelling, vol.~31, 2000,
  pp.~75--81.

\bibitem{Hauck2010}
{\sc C.~D. Hauck}, {\em {High-order entropy-based closures for linear transport
  in slab geometry}}, Communications in Mathematical Sciences, 9 (2011),
  pp.~187--205.

\bibitem{HenIzaSie06}
{\sc H.~Hensel, R.~Iza-Teran, and N.~Siedow}, {\em {Deterministic model for
  dose calculation in photon radiotherapy}}, Physics in medicine and biology,
  51 (2006), pp.~675--693.

\bibitem{Hillen2013}
{\sc T.~Hillen and K.~J. Painter}, {\em {Transport and anisotropic diffusion
  models for movement in oriented habitats}}, Lecture Notes in Mathematics,
  2071 (2013), pp.~177--222.

\bibitem{Jea17}
{\sc J.~H. Jeans}, {\em {The equations of radiative transfer of energy}},
  Monthly Notices Royal Astronomical Society, 78 (1917), pp.~28--36.

\bibitem{Jun00}
{\sc M.~Junk}, {\em {Maximum entropy for reduced moment problems}}, Math. Meth.
  Mod. Appl. Sci., 10 (2000), pp.~1001--1025.

\bibitem{Ker76}
{\sc D.~S. Kershaw}, {\em {Flux Limiting Nature's Own Way: A New Method for
  Numerical Solution of the Transport Equation}}, Lawrence Livermore National
  Laboratory, UCRL-78378,  (1976).

\bibitem{Levermore1996}
{\sc C.~D. Levermore}, {\em {Moment closure hierarchies for kinetic theories}},
  Journal of Statistical Physics, 83 (1996), pp.~1021--1065.

\bibitem{Levermore1998}
\leavevmode\vrule height 2pt depth -1.6pt width 23pt, {\em {Moment Closure
  Hierarchies for the Boltzmann-Poisson Equation}}, VLSI Design, 6 (1998),
  pp.~97--101.

\bibitem{Lewis-Miller-1984}
{\sc E.~E. Lewis and J.~{W. F. Miller}}, {\em {Computational Methods in Neutron
  Transport}}, John Wiley and Sons, New York, 1984.

\bibitem{Min78}
{\sc G.~N. Minerbo}, {\em {Maximum entropy Eddington factors}}, J. Quant.
  Spectrosc. Radiat. Transfer, 20 (1978), pp.~541--545.

\bibitem{Monreal2008}
{\sc P.~Monreal and M.~Frank}, {\em {Higher order minimum entropy
  approximations in radiative transfer}}, arXiv preprint arXiv:0812.3063,
  (2008), pp.~1--18.

\bibitem{Pom92}
{\sc G.~C. Pomraning}, {\em {The Fokker-Planck operator as an asymptotic
  limit}}, Math. Mod. Meth. Appl. Sci., 2 (1992), pp.~21--36.

\bibitem{Schneider2016b}
{\sc F.~Schneider}, {\em {Implicit-explicit, realizability-preserving
  first-order scheme for moment models with Lipschitz-continuous source
  terms}}, arXiv preprint,  (2016).

\bibitem{Schneider2015}
\leavevmode\vrule height 2pt depth -1.6pt width 23pt, {\em {Kershaw closures
  for linear transport equations in slab geometry I: Model derivation}},
  Journal of Computational Physics, 322 (2016), pp.~905--919.

\bibitem{Schneider2016a}
\leavevmode\vrule height 2pt depth -1.6pt width 23pt, {\em {Kershaw closures
  for linear transport equations in slab geometry II: high-order
  realizability-preserving discontinuous-Galerkin schemes}}, Journal of
  Computational Physics, 322 (2016), pp.~920--935.

\bibitem{Schneider2016}
\leavevmode\vrule height 2pt depth -1.6pt width 23pt, {\em {Moment models in
  radiation transport equations}}, Dr. Hut Verlag, mathematik~ed., 2016.

\bibitem{Schneider2014}
{\sc F.~Schneider, G.~W. Alldredge, M.~Frank, and A.~Klar}, {\em {Higher Order
  Mixed-Moment Approximations for the Fokker--Planck Equation in One Space
  Dimension}}, SIAM Journal on Applied Mathematics, 74 (2014), pp.~1087--1114.

\bibitem{Schneider2015b}
{\sc F.~Schneider, J.~Kall, and G.~Alldredge}, {\em {A realizability-preserving
  high-order kinetic scheme using WENO reconstruction for entropy-based moment
  closures of linear kinetic equations in slab geometry}}, Kinetic and Related
  Models, 9 (2015), pp.~193--215.

\bibitem{Schneider2015c}
{\sc F.~Schneider, J.~Kall, and A.~Roth}, {\em {First-order quarter- and
  mixed-moment realizability theory and Kershaw closures for a Fokker-Planck
  equation in two space dimensions}}, Kinetic and Related Models, to appear
  (2016).

\end{thebibliography}

\end{document}